\documentclass[11pt]{amsart}

\usepackage{amsfonts}
\usepackage{amssymb}
\usepackage{hyperref}

\theoremstyle{plain}
\newtheorem{thm}{Theorem}[section]

\newtheorem{cor}[thm]{Corollary}
\newtheorem{prop}[thm]{Proposition}

\newtheorem{prob}[thm]{Problem}

\theoremstyle{definition}

\theoremstyle{remark}

\newcommand{\lemref}[1]{\hyperref[#1]{Lemma \ref*{#1}}}
\newcommand{\thmref}[1]{\hyperref[#1]{Theorem \ref*{#1}}}
\newcommand{\propref}[1]{\hyperref[#1]{Proposition \ref*{#1}}}
\newcommand{\corref}[1]{\hyperref[#1]{Corollary \ref*{#1}}}
\newcommand{\defref}[1]{\hyperref[#1]{Definition \ref*{#1}}}
\newcommand{\remref}[1]{\hyperref[#1]{Remark \ref*{#1}}}
\newcommand{\conjref}[1]{\hyperref[#1]{Conjecture \ref*{#1}}}
\newcommand{\probref}[1]{\hyperref[#1]{Problem \ref*{#1}}}

\makeatletter
\newcommand*{\defeq}{\mathrel{\rlap{%
                     \raisebox{0.27ex}{$\m@th\cdot$}}%
                     \raisebox{-0.27ex}{$\m@th\cdot$}}%
                     =}

\numberwithin{equation}{section}
\makeatother

\makeatletter
\def\@setcopyright{}
\def\serieslogo@{}
\makeatother

\begin{document}

\title{Squarefree polynomials with prescribed coefficients}

\author{Amotz Oppenheim}
\address{Raymond and Beverly Sackler School of Mathematical Sciences, Tel-Aviv University, Tel-Aviv, Israel}
\email{amotzo@mail.tau.ac.il}

\author{Mark Shusterman}
\address{Raymond and Beverly Sackler School of Mathematical Sciences, Tel-Aviv University, Tel-Aviv, Israel}
\email{markshus@mail.tau.ac.il} 
\date{}
\begin{abstract}
For nonempty subsets $S_0, \dots, S_{n-1}$ of a (large enough) finite field $\mathbb{F}$ satisfying  
$$|S_1|, \dots, |S_{n-1}| > 2 \quad \mathrm{or} \quad |S_1|,|S_{n-1}| > n - 1,$$
we show that there exist $a_0 \in S_0, \dots, a_{n-1} \in S_{n-1}$ such that
$$ T^n + a_{n-1}T^{n-1} + \dots + a_1T + a_0 \in \mathbb{F}[T] $$
is a squarefree polynomial.

\end{abstract}

\maketitle

\section{Introduction}

Consider the positive $(n+1)$-digit integers 
\begin{equation}
m = a_0 + a_1q + \dots + a_{n-1}q^{n-1} + a_nq^n
\end{equation}
in a certain fixed base $q$.
A very interesting (and very broad) problem in analytic number theory is to study the behavior of an arithmetic function over the integers $m$ whose vector of digits lies in some `box'
\begin{equation}
B \defeq 
\big\{(a_0, \dots, a_n) \in S_0 \times \dots \times S_n \ | \ S_i \subseteq \{0,1, \dots, q-1\}, \ 0 \leq i \leq n \big\}.
\end{equation}
Boxes include (unions of) arithmetic progressions, short intervals, base-$5$ integers not having $3$ as a digit, and many other examples.
This generalized point of view has its origins in the works \cite{Sierp, Sierp2} of Sierpinski who studied primes with prescribed values of $a_0$ and $a_n$.
Many works followed, for instance \cite{Kon} by Konyagin, \cite{Tao} by Tao, and also \cite{BoDi, Bugea, BK, CM, CM2, EMS2, GShpar, Harm, HK, Kon, KMS, Meng, Shpar, Wol}.
Furthermore, Bourgain has shown in \cite{Bourg} (improving on several previous results) that one can find a prime with a fixed proportion of the digits prescribed, and Maynard in \cite{May} has shown that there are primes in every large cube, that is, in a box where for all $0 \leq i \leq n$, $S_i = S$ is a set with a small (compared to $q$) complement.

As can be seen from the works mentioned above, von Mangoldt is not the only function studied in this context.
For example, Erdos, Mauduit, and Sarkozy posed the following problem in \cite{EMS}.
\begin{prob} \label{ErdProb}

Is there (for infinitely many values of $n$) a squarefree integer in the box given by $S_i \defeq \{0,1\}$ for all $0 \leq i \leq n -1$ and $S_n = \{1\}$?

\end{prob}
The problem has been solved affirmatively for small values of $q$ (and variants of it were considered) in \cite{EMS}, 
in \cite{FK} by Filaseta and Konyagin,
and in \cite{BShpar} by Banks and Shparlinski.
Even if each $S_i$ is allowed to contain some fixed amount $k > 2$ of elements, the problem remains open.
In a slightly different vein, it is shown in the recent work \cite{DES} of Dietmann, Elsholtz, and Shparlinski that one can find a squarefree integer with almost $40\%$ of the (base-$2$) digits prescribed. 

The function field analog of our setup,
where the digits are replaced by the coefficients of a polynomial
\begin{equation}
m(T) = a_0 + a_1T + \dots + a_{n-1}T^{n-1} + a_nT^n
\end{equation} 
over a finite field $\mathbb{F}_q$, has also received a lot of attention, as can be seen from \cite{AGGMY, BBF, BBR, CE, CMRSS, CMS, Co, FY, Gar, HKR, Ha, Hay, HM, Hsu, KR1, KR2, KRRR, MR, Mos, PT, Pol, Ree, RMKR, Rodit, Rodg0, Rodg, Shpar0, TW, vO, Wan, YM}.
In this work we contribute to the study of the function field analog by giving lower bounds on the number of squarefrees in `sparse' boxes.
By saying that a box $S_0 \times \dots \times S_n$ is sparse, 
we roughly mean that each $S_i$ (for $0 \leq i \leq n$) is very small compared to $q$. 
Our first result solves a function field analog of \probref{ErdProb}.
\begin{thm} \label{FirstRes}
Let $n > 1$ be an integer, and let $S_0, \dots, S_{n}$ be subsets of a finite field $\mathbb{F}$ with
\begin{equation} \label{oddCond}
|S_0|, |S_n| = 1, \quad S_n \neq \{0\}, \quad |S_1|, \dots, |S_{n-1}| = 3.
\end{equation}
Then there exist $a_0 \in S_0, \dots, a_{n} \in S_{n}$ such that
\begin{equation}
m(T) = a_0 + a_1T + \dots + a_{n-1}T^{n-1} + a_nT^n
\end{equation}
is a squarefree polynomial over $\mathbb{F}$.
Moreover, if the characteristic of $\mathbb{F}$ is $2$, we can replace the last condition in \eqref{oddCond} by the weaker assumption that
\begin{equation}
|S_1|, \dots, |S_{n-1}| = 2.
\end{equation}
In particular, for any distinct $0 \neq a, b \in \mathbb{F}$ we can take $S_i \defeq \{a,b\}$ for all $0 \leq i \leq n-1$ and $S_n = \{a\}$ as in \probref{ErdProb}.
\end{thm}

A variety of methods and techniques have been employed in the study of problems similar to \probref{ErdProb} (both over the integers and over rings of polynomials with coefficients from a finite field). Some of these are: circle method (and other `Fourier techniques'), Riemann Hypothesis (and zero-free regions), exponential and character sums, geometry of numbers, sieves, trace formulas, Chebotarev's density theorems, and equidistribution results in the large finite field limit. 
As the boxes considered in \thmref{FirstRes} are quite sparse, (especially when $|\mathbb{F}|$ is large) it seems that the aforementioned techniques are not sufficient here.
Our proof of \thmref{FirstRes} relies on the Combinatorial Nullstellensatz which has not yet been applied in function field arithmetic.

Using the Combinatorial Nullstellensatz we also obtain a variant of \thmref{FirstRes}
that guarantees the existence of squarefrees in even sparser boxes (when the cardinality of the field is large compared to the degree).

\begin{thm} \label{SecRes}
Let $\mathbb{F}$ be a finite field of characteristic $p$,
let $n > 2$ be an integer not congruent to $2$ mod $p$, 
and let $S_0, \dots, S_{n}$ be subsets of $\mathbb{F}$ with
\begin{equation} \label{CharNo2Eq}
|S_1|, |S_{n-1}| = n, \quad \forall i \notin\{1,n-1\} \ |S_i| = 1, \quad S_n \neq \{0\}.
\end{equation}
Then there exist $a_0 \in S_0, \dots, a_{n} \in S_{n}$ such that
\begin{equation}
m(T) = \sum_{i=0}^n a_iT^i
\end{equation}
is a squarefree polynomial over $\mathbb{F}$.
\end{thm}

This theorem is of `large finite field limit' type since if we fix $\mathbb{F}$, 
there are only finitely many values of $n$ to which \thmref{SecRes} is applicable, 
while if we fix $n$ and take $|\mathbb{F}|$ to infinity, 
we find squarefrees in very sparse boxes, 
as compared to results on squarefrees in short intervals such as \cite[Theorem 1.3 (ii)]{KR2} by Keating and Rudnick, and \cite[Theorem 2.4]{CE} by Carmon. 
As in \thmref{FirstRes}, the first condition in \eqref{CharNo2Eq} can be weakened in case $p=2$.

Combining our results with those of Erdos from \cite{Erd}, we show that (in the large finite field limit) almost every polynomial in a cube is squarefree.

\begin{cor} \label{cor}

Fix an integer $n > 1$.
For every finite field $\mathbb{F}$ pick subsets $S_0 = S_0(\mathbb{F}), \dots, S_n = S_n(\mathbb{F}) \subseteq \mathbb{F}$ having the same cardinality $C(\mathbb{F})$ such that 
\begin{equation} \label{GrowthEq}
\lim_{|\mathbb{F}| \to \infty} C(\mathbb{F}) = \infty.
\end{equation}
Then
\begin{equation*}
\lim_{|\mathbb{F}| \to \infty} \frac{\# \big\{(a_0, \dots, a_n) \in S_0 \times \dots \times S_n 
: \sum_{i=0}^n a_iT^i \ \text{is squarefree} \big\}}
{\# (S_0 \times \dots \times S_n)} = 1.
\end{equation*}

\end{cor}

Note that the assumption in \eqref{GrowthEq} is necessary, and (perhaps surprisingly) an arbitrarily slow growth rate suffices, as the corollary shows.
For sufficiently high growth rates in \eqref{GrowthEq}, \corref{cor} can be proved using the standard sieving and Fourier transform techniques.
As the example following \cite[Theorem 1.3]{KR2} shows, the analog of \corref{cor} for general boxes (with growing volume) fails, so the restriction in \corref{cor} to cubes is necessary.

\section{The proofs}   

In order to conclude that a degree $n \geq 2$ polynomial
\begin{equation}
m(T) = A_0 + A_1T + \dots + A_{n-1}T^{n-1} + A_nT^n \in \mathbb{F}[T]
\end{equation}
over a field $\mathbb{F}$ is squarefree, it suffices to show that its discriminant, 
which is a homogeneous polynomial in $A_0, \dots, A_n$ of degree $2n-2$, does not vanish. 
For that, we use the following special case of the Combinatorial Nullstellensatz \cite[Theorem 1.2]{Alon}.

\begin{cor} \label{NullCor}

Let $\mathbb{F}$ be a field, let $f(A_0, \dots, A_n) \in \mathbb{F}[A_0, \dots, A_n]$ be a homogeneous polynomial,
and let $A_0^{i_0}A_1^{i_1} \cdots A_n^{i_n}$ be a monomial that appears (with a nonzero coefficient) in $f$.
Suppose that $S_0, S_1, \dots, S_n \subseteq \mathbb{F}$ satisfy
\begin{equation}
|S_0| > i_0, \ |S_1| > i_1, \dots, \ |S_n| > i_n.
\end{equation}
Then there exist $a_0 \in S_0, \ a_1 \in S_1, \dots, \ a_n \in S_n$ such that
\begin{equation}
f(a_0, a_1, \dots, a_n) \neq 0.
\end{equation}

\end{cor}

Hence, in order to obtain the first part of \thmref{FirstRes}, it suffices to establish the following claim.

\begin{prop} \label{FirstProp}

The monomial 
\begin{equation}
A_1^{2}A_2^{2} \cdots A_{n-1}^{2}
\end{equation}
appears (with coefficient $\pm 1$) in the degree $(2n-2)$ homogeneous polynomial
\begin{equation}
f(A_0, A_1, \dots, A_n) \defeq \mathrm{Disc}_T(A_0 + A_1T + \dots + A_{n}T^n).
\end{equation}

\end{prop}

\begin{proof}

We induct on $n$ noting that the base case $n = 2$ is well known. 
Setting 
\begin{equation} \label{qDefEq}
p(T) \defeq A_1T + A_2T^2 + \dots + A_{n}T^n, \quad q(T) \defeq \frac{p(T)}{T},
\end{equation}
we observe that it is sufficient to show that $\pm A_1^{2}A_2^{2} \cdots A_{n-1}^{2}$ appears in
\begin{equation}
f\big(0,A_1, \dots, A_n\big) = \mathrm{Disc}_T\big(p(T)\big).
\end{equation}
The basic relation between the resultant and the discriminant tells us that
\begin{equation} \label{BasEq}
\begin{split}
\mathrm{Disc}_T\big(p(T)\big) &= \frac{(-1)^{\frac{n(n-1)}{2}}\mathrm{Res}_T\big(p(T),p'(T)\big)}{A_n}.
\end{split}
\end{equation}
We can ignore the sign, use the notation of \eqref{qDefEq}, and differentiate with respect to $T$ to obtain
\begin{equation} \label{DerEq}
\frac{\mathrm{Res}_T\big(p(T),p'(T)\big)}{A_n} = \frac{\mathrm{Res}_T\big(Tq(T),q(T) + Tq'(T)\big)}{A_n}.
\end{equation}
Using the multiplicativity of the discriminant, we see that the right-hand side of \eqref{DerEq} equals
\begin{equation} \label{DiffEq}
\frac{\mathrm{Res}_T\big(T,q(T) + Tq'(T)\big) \cdot \mathrm{Res}_T\big(q(T),q(T) + Tq'(T)\big)}{A_n}.
\end{equation}
Using the invariance rule of the resultant under linear transformations, we conclude that up to a sign,
\eqref{DiffEq} equals
\begin{equation} \label{TransEq}
\frac{\mathrm{Res}_T\big(T,q(T) \big) \cdot \mathrm{Res}_T\big(q(T), Tq'(T)\big)}{A_n}.
\end{equation} 
Another application of multiplicativity shows that (up to a sign) \eqref{TransEq} is
\begin{equation} \label{AnEq}
\frac{\mathrm{Res}_T\big(T,q(T) \big)^2 \cdot \mathrm{Res}_T\big(q(T), q'(T)\big)}{A_n}.
\end{equation}
Applying \eqref{BasEq} to $q(T)$ we find that (up to a sign) \eqref{AnEq} equals
\begin{equation} \label{FinEq}
\mathrm{Res}_T\big(T,q(T) \big)^2 \cdot \mathrm{Disc}_T\big(q(T)\big).
\end{equation}
Since $q(T)$ is congruent to $A_1$ mod $T$, we have
\begin{equation} \label{QuadFacEq}
\mathrm{Res}_T\big(T,q(T) \big)^2 = \mathrm{Res}_T\big(T, A_1 \big)^2 = A_1^2,
\end{equation} 
and from induction, we get that $\pm A_2^2 \cdots A_{n-1}^2$ appears in $\mathrm{Disc}_T\big(q(T)\big)$.
Hence, $\pm A_1^2 \cdot A_2^2 \cdots A_{n-1}^2$ appears in \eqref{FinEq} as required.
\end{proof}

If the characteristic of $\mathbb{F}$ is 2, the polynomial
\begin{equation} \label{DefDiscEq}
f(A_0, A_1, \dots, A_n) \defeq \mathrm{Disc}_T(A_0 + A_1T + \dots + A_{n}T^n)
\end{equation}
is easily seen to be a square of some $g(A_0, A_1, \dots, A_n) \in \mathbb{F}[A_0, \dots, A_n]$,
so we can apply \corref{NullCor} to $g$ instead.
Since $A_1^2 \cdots A_{n-1}^2$ appears in $f$ (with coefficient 1), 
it readily follows that $A_1 \cdots A_{n-1}$ appears in $g$, and the second part of \thmref{FirstRes} is thus established.

Similarly, in order to prove \thmref{SecRes}, it suffices to know the following.

\begin{prop}
The monomial $A_1^{n-1}A_{n-1}^{n-1}$ appears with coefficient 
\begin{equation}
\pm (n-2)^{n-2}
\end{equation}
in the polynomial $f$ from \eqref{DefDiscEq}.

\end{prop}

\begin{proof}

As in the proof of \propref{FirstProp}, we just need to find the coefficient of our monomial in
\begin{equation} \label{SparseDiscEq}
f\big(0, A_1, 0, \dots, 0, A_{n-1}, A_n\big) = \mathrm{Disc}_T(A_1T + A_{n-1}T^{n-1} + A_nT^n).
\end{equation}
By the multiplicativity of the discriminant, we see that the right-hand side of \eqref{SparseDiscEq} equals (up to a sign)
\begin{equation}
\begin{split}
&\mathrm{Disc}_T(T) \cdot \mathrm{Disc}_T(A_1 + A_{n-1}T^{n-2} + A_nT^{n-1}) \cdot \\
&\mathrm{Res}_T(T, A_1 + A_{n-1}T^{n-2} + A_nT^{n-1})^2
\end{split}
\end{equation}
and this reduces (using the properties of resultants mentioned above) to
\begin{equation} \label{TrinEq}
A_1^2\mathrm{Disc}_T(A_1 + A_{n-1}T^{n-2} + A_nT^{n-1}).
\end{equation}
It is shown in \cite{Left} that \eqref{TrinEq} evaluates (up to a sign) to
\begin{equation}
A_1^{n-1}\big((n-1)^{n-1}A_1A_n^{n-2} \pm (n-2)^{n-2}A_{n-1}^{n-1} \big)
\end{equation}
so we have found the required coefficient.
\end{proof}

Let us now deduce \corref{cor} from \thmref{FirstRes}.

\begin{proof}

Define an $(n+1)$-uniform $(n+1)$-partite hypergraph $H$ whose set of vertices is the disjoint union of 
$S_0, \dots, S_n$ and for any $a_0 \in S_0, \dots, a_n \in S_n$ the set $\{a_0, \dots, a_n\}$ is an edge if 
\begin{equation}
m(T) = a_0 + a_1T + \dots + a_{n-1}T^{n-1} + a_nT^n
\end{equation}
is not squarefree. 
\thmref{FirstRes} tells us that $H$ does not contain the complete $(n+1)$-uniform $(n+1)$-partite hypergraph 
$K^{(n+1)}(3, \dots, 3)$, 
so by \cite[Theorem 1]{Erd}, the number of edges in $H$ is at most
\begin{equation}
\big(2nC(\mathbb{F})\big)^{n+1-\frac{1}{3^n}}
\end{equation} 
when $|\mathbb{F}|$ is large enough (compared to $n$).
Hence, the density of the polynomials in $S_0 \times \dots \times S_n$ that are not squarefree is at most
\begin{equation} \label{DensityEq}
\frac{\big(2nC(\mathbb{F})\big)^{n+1-\frac{1}{3^n}}}{C(\mathbb{F})^{n+1}} \leq
\frac{(2n)^{2n}}{C(\mathbb{F})^{\frac{1}{3^n}}}.
\end{equation}
As $|\mathbb{F}| \to \infty$, $C(\mathbb{F})$ grows to $\infty$, so the density in \eqref{DensityEq} goes to $0$. 
\end{proof}

\section*{Acknowledgments}
We sincerely thank Dan Carmon, Gal Dor, and Noga Alon for helpful discussions.
Mark Shusterman is grateful to the Azrieli Foundation for the award of an Azrieli Fellowship.
The second author was partially supported by a grant of the Israel Science Foundation with cooperation of UGC no. 40/14.

\end{document}